\documentclass[a4paper,reqno,12pt]{amsart}   
\usepackage{newtxtext} 
\usepackage{amsmath}
\usepackage[bigdelims]{newtxmath}
\usepackage[T1]{fontenc}
\usepackage{textcomp}

\usepackage{tikz-cd}
\tikzcdset{
	arrow style=tikz,
	diagrams={>=stealth}
}

\usepackage[mathscr]{eucal}
\usepackage{mathrsfs}

\usepackage{xcolor}

\def\timestring{\begingroup
	\count0 = \time \divide\count0 by 60
	\count2 = \count0	
	\count4 = \time \multiply\count0 by 60
	\advance\count4 by -\count0 	
	\ifnum\count4<10 	\toks1 = {0}
	\else				\toks1 = {}
	\fi
	\ifnum\count2<12 	\toks0 = {a.m.}
	\else				\toks0 = {p.m.}
		\advance\count2 by -12
	\fi
	\ifnum\count2=0 \count2 = 12 \fi
	\number\count2:\the\toks1 \number\count4 \thinspace \the\toks0  
\endgroup}

\numberwithin{equation}{section}	

\theoremstyle{plain}
\newtheorem{theorem}{Theorem}[section]      
\newtheorem{lemma}[theorem]{Lemma}
\newtheorem{proposition}[theorem]{Proposition}

\newtheorem*{theorem*}{Theorem}
\newtheorem*{facts*}{Facts}

\theoremstyle{definition}
\newtheorem{definition}[theorem]{Definition}

\theoremstyle{remark}
\newtheorem{remark}[theorem]{Remark}

\def\geqsl{\geqslant}

\newcommand{\trace}[1]{\ensuremath{\operatorname{tr}{#1}}}

\newcommand{\tp}[1]{\ensuremath{\hspace{1truept}\vphantom{{#1}}^{t}\hspace{-1.8truept}#1}}

\newcommand{\Mat}[1]{\ensuremath{\operatorname{Mat}_{#1}}}
   
\newcommand{\diag}{\ensuremath{\operatorname{diag}}}

\newcommand{\Lie}[1]{\ensuremath{\operatorname{Lie}{#1}}}

\newcommand{\Ad}{\ensuremath{\operatorname{Ad}}}  

\newcommand{\R}{\ensuremath{\mathbb R}}
\newcommand{\C}{\ensuremath{\mathbb C}}

\newcommand{\g}{\ensuremath{\mathfrak{g}}}   
\newcommand{\q}{\ensuremath{\mathfrak{q}}}   

\renewcommand{\u}{\ensuremath{\mathfrak{u}}}

\newcommand{\GL}[1]{\ensuremath{\mathrm{GL}_{#1}}}  
\newcommand{\SL}[1]{\ensuremath{\mathrm{SL}_{#1}}}  
\newcommand{\U}{\ensuremath{\mathrm{U}}}            
\newcommand{\SU}{\ensuremath{\mathrm{SU}}}          
\newcommand{\Sp}{\ensuremath{\mathrm{Sp}}}          

\newcommand{\Gc}{\ensuremath{G_{\! c}}}

\DeclareMathOperator{\re}{Re}		
\DeclareMathOperator{\im}{Im}		

\newcommand{\rmd}{\ensuremath{\operatorname{d}\!}}
\newcommand{\pd}{\ensuremath{\partial}}

\newcommand{\norm}[1]{\ensuremath{\left\|#1\right\|}}

\def\<{\langle}
\def\>{\rangle}

\renewcommand{\vec}[1]{\ensuremath{\boldsymbol #1}}

\def\c2vec#1#2{ %
   \left[ \begin{smallmatrix} %
           #1 \\ #2  \end{smallmatrix} %
   \right]}

\DeclareMathOperator{\End}{End}

\newcommand{\ai}{\ensuremath{\sqrt{-1}}}
\renewcommand{\rmd}{\ensuremath{\mathrm{d}}}
\newcommand{\del}{\ensuremath{\pd}}
\newcommand{\antidel}{\ensuremath{\overline{\pd}}}
\newcommand{\abs}[1]{\ensuremath{\left|\, #1 \, \right|}}
\renewcommand{\tp}[1]{\ensuremath{\hspace{1truept}\vphantom{{#1}}^{t}\hspace{-1.6truept}#1}}

\newcommand{\CP}{\ensuremath{ {\mathbb{CP}} }}
\renewcommand{\Gc}{\ensuremath{G^c}}
\newcommand{\Gu}{\ensuremath{G}}

\newcommand{\ourZ}{\ensuremath{\mathit{Z}}}
\newcommand{\ourLSmfd}{\ensuremath{\mathscr{L}}}
\newcommand{\ourorbit}{\ensuremath{\mathscr O}^c_{\!\lambda}}		
\newcommand{\Tc}{\ensuremath{T^{\C}}} 

\makeatletter
\@namedef{subjclassname@2020}{\textup{2020} Mathematics Subject Classification}
\makeatother

\begin{document}

\title[A hyperk\"ahler metric on twisted cotangent bundle of $\CP^n$]
{A hyperk\"ahler metric on twisted cotangent bundles of the complex projective space}

\author{Takashi Hashimoto}
%
\address{
	Center for Data Science Education, 
	Tottori University, 
	4-101, Koyama-Minami, Tottori, 680-8550, Japan.
}
\email{thashi@tottori-u.ac.jp}
\date{\today,\timestring}
\keywords{%
hyperk{\"a}hler metric, Ricci-flat metric, twisted cotangent bundle, complex coadjoint orbit, moment map%
}
\subjclass[2020]{Primary 53C26}

\begin{abstract}
We construct a hyperk\"ahler metric on twisted cotangent bundles of the complex projective space $\CP^n$ explicitly in terms of local coordinates.
Note that the twisted cotangent bundles of $\CP^n$ are holomorphically isomorphic to complex semisimple coadjoint orbits of $\SL{n+1}(\C)$, as we showed in \cite{equiv_sympl}.
\end{abstract}

\maketitle

\section{Introduction}

It is well-known that complex coadjoint orbits possess a hyperk\"ahler structure (see e.g.~\cite{Hitchin92_bourbaki} \cite{Kronheimer90}).
In fact, Kronheimer constructed a hyperk\"ahler structure on complex semisimple coadjoint orbits in \cite{Kronheimer90} via identification of the orbits with moduli spaces of instantons on $\R^4$ whose quaternionic structure induces the hyperk\"ahler structure thereof.
Then, in \cite{BiquardGauduchon96}, Biquard and Gauduchon gave an explicit formula for the potential of the hyperk\"ahler metric on the complexification of a compact Hermitian symmetric orbit.

The aim of this paper is to provide a direct construction of a hyperk\"ahler metric on twisted cotangent bundles of the complex projective space $\CP^n$ in terms of local coordinates.
We note that our twisted cotangent bundles of $\CP^n$ are holomorphically isomorphic to complex semisimple coadjoint orbits of $\SL{n+1}(\C)$ (see \cite{equiv_sympl}).

To be more precise, let us consider the coadjoint action of a linear connected complex reductive Lie group $\Gc$ on the dual $\g^*$ of its Lie algebra $\g:=\Lie(\Gc)$.
Then, $\Gc$-orbits of semisimple elements $\lambda \in \g^*$ under the coadjoint action, denoted by $\ourorbit$ in this paper, are holomorphically isomorphic to twisted cotangent bundles of flag varieties $\Gc / Q$, which are affine bundles on $\Gc / Q$, where $Q$ is a parabolic subgroup of $\Gc$ whose Levi factor coincides with the isotropy subgroup of $\lambda$ (see \cite{Donaldson02_holomorphic_discs} \cite{equiv_sympl}).
Furthermore, the isomorphisms are given by the moment maps with respect to the Kirillov-Kostant-Souriau form on $\ourorbit$ and the holomorphic symplectic form on the twisted cotangent bundles inherited from the canonical one on the cotangent bundles.
Note that the twisted cotangent bundles are locally isomorphic to the cotangent bundles, and that if $\lambda$ tends to $0$ then the twisted cotangent bundles reduce to the cotangent bundles.
Now, in this paper, we introduce local coordinates via the holomorphic isomorphism given by the moment map in the case where $\Gc=\SL{n+1}(\C)$ and $Q$ is its maximal subgroup such that $\Gc / Q$ is isomorphic to $\CP^n$, 
and construct a hyperk\"ahler metric on the twisted cotangent bundle of $\CP^n$ explicitly in terms of local coordinates (see below for more details about our setting).

Our strategy is as follows.
First, we seek for Ricci-flat metrics since hyperk\"ahler manifolds are Ricci-flat.
Just as Eguchi and Hanson obtained the Eguchi-Hanson metric by solving an ordinary differential equation (ODE) in \cite{EguchiHanson79} (see also \cite{Calabi79}), 
we shall reduce the Ricci-flatness condition to a first-order ODE, whose independent variable is a $\Gu$-invariant function described with the holomorphic moment map mentioned above, where $\Gu = \SU(n+1)$, the compact real form of $\Gc=\SL{n+1}(\C)$;
we solve the ODE explicitly.
Next, we find a hyperk\"ahler metric among the Ricci-flat ones obtained in the first step.
We will see that there is one and only one hyperk\"ahler metric among them.

The plan of this paper is as follows.
In the rest of this section, we describe our setup in detail, and fix the notations.
In Section 2, 
we recall the construction of twisted cotangent bundles of $\Gc / Q$ and the isomorphisms between them and the complex coadjoint orbits $\ourorbit$ given by moment maps.
We remark that the parabolic subgroup $Q$ in this paper is maximal, though the construction of the twisted cotangent bundles can be applied to the cases where parabolic subgroups are not necessarily maximal.
Then we introduce the $\Gu$-invariant function, which we denote by $\tau$; it is defined to be the norm squared of the holomorphic moment map translated by a constant.
In Section 3, 
we find Ricci-flat K\"ahler metrics on the twisted cotangent bundle.
If we assume that the primitive $1$-form of the associated fundamental $2$-form is described in terms of $\tau$ as in \eqref{e:one-form} below, then we will see that the Ricci-flatness condition reduces to a first-order ODE (Theorem \ref{t:ricci-flat}), as mentioned above.
In Section 4,
we shall show that the hyperk\"ahlerian condition is equivalent to the fact that the Hermitian matrix associated with the metric belongs to the complex symplectic group (Proposition \ref{p:J^2 and K^2}), and we pick a hyperk\"ahler metric out of the Ricci-flat metrics obtained in Section 3, which is our main result (Theorem \ref{t:hk_sol}).
Finally, in Appendix, we collect a few formulas from linear algebra, mainly about matrices of rank 2.
%
%

\subsection{}

Now we introduce our setup.

Throughout the paper, we let $\Gu:=\SU(n+1)$ and $\Gc:=\SL{n+1}(\C)$ and denote their Lie algebras by $\g_0:=\mathfrak{su}(n+1)$ and $\g:=\mathfrak{sl}_{n+1}(\C)$ for brevity.
We often identify the dual space $\g^{*}$ with $\g$ by the trace form $B$:
\begin{equation}
\label{e:isom_g_and_gstar}
	B(X,Y) = \trace{(X Y)} \quad (X,Y \in \g).
\end{equation}
Note that the form $B$ is negative definite on $\g_{0}=\mathfrak{su}(n+1)$.
Let $Q$ be a maximal parabolic subgroup of $\Gc$ given by
\[
	Q =\left\{ \begin{bmatrix} a & 0 \\ c & d \end{bmatrix} \in \Gc ; a \in \GL{n}(\C), c \in \Mat{1 \times n}(\C), d \in \C^{\times} \right\},
\]
and $\mathfrak q$ the Lie algebra of $Q$ given by
\[
	\mathfrak q = \left\{ \begin{bmatrix} a & 0 \\ c & d \end{bmatrix} \in \g ; a \in \Mat{n \times n}(\C), c \in \Mat{1 \times n}(\C), d \in \C \right\}.
\]
Thus $\Gc/Q$ is isomprphic to the comlex projective space 
\[
	\CP^n=\big\{ [\zeta_0: \zeta_1: \dots :\zeta_n]; \tp{(\zeta_0,\zeta_1,\dots,\zeta_n)} \in \C^{n+1} \backslash \{0\} \big\}
\]
by $g.e_{Q} \leftrightarrow [g.e_{n}]$, where $e_{Q}$ is the origin of $\Gc/Q$ and $e_{n}=\tp{(0,\dots,0,1)}$.

Let $\lambda \in \g^{*}$ be a character (i.e.~one-dimensional representation) of $\mathfrak q$ which sends 
\( \left[\begin{smallmatrix} a & 0 \\ c & d \end{smallmatrix}\right] \) to $-s d$ 
for $s \in \C$,
and $\lambda^{\vee}$ the corresponding element of $\g$ under the identification $\g^{*} \simeq \g$ via the bilinear form $B$ given by \eqref{e:isom_g_and_gstar}:  
\[
	\lambda^{\vee} = \frac{s}{n+1} \begin{bmatrix} 1_{n} & 0 \\[2pt] 0 & - n \end{bmatrix} \in \g. 
\]
The $\Gc$-orbit $\ourorbit:=\Gc \cdot \lambda$ of $\lambda$ under the coadjoint action is canonically isomorphic to $\Gc / L$, where
\begin{equation}
	L = \left\{ \begin{bmatrix} a & 0 \\ 0 & d \end{bmatrix} \in \Gc ; a \in \GL{n}(\C), d \in \C^{\times} \right\}
\end{equation}
is the Levi factor of $Q$.
Let $\mathfrak l$ be the Lie algebra of $L$, $\u^{-}$ the nilradical of $\q$, and $\u$ the opposite of $\u^{-}$.
Then we have the following decompositions:
\begin{equation}
\label{e:triangular_decomp}
	\q = \mathfrak l \oplus \u^{-}   \quad \text{and} \quad  \g = \u \oplus \mathfrak l \oplus \u^{-}.
\end{equation}
Note that $\u$ and $\u^{-}$ are abelian and isomorphic to $\C^{n}$. 
Denoting the analytic subgroups of $\u^{-}$ and $\u$ by $U^{-}$ and $U$ respectively, we see that
\begin{equation}
	Q = L U^{-} = U^{-} L,
\end{equation}
and that $\Gc$ has an open covering: %
\begin{equation}
\begin{aligned}
	\Gc &= \bigcup_{\sigma \in W/W_{\lambda}} \sigma U Q, 
\end{aligned}
\end{equation}
where $W$ is the Weyl group of $\Gc$ and $W_{\lambda}$ the isotropy subgroup of $\lambda$ in $W$. 
We fix a representative $\dot \sigma \in \Gu$ of each element $\sigma \in W/W_{\lambda}$ once and for all, and identify $\sigma$ with $\dot \sigma$.
In fact, as the representatives, we can take the signed permutation matrices of determinant one that correspond to cyclic permutations $(\alpha,\alpha+1,\dots,n)$,
which we denote by $\dot \sigma_{\alpha}$ for $\alpha=0,1,\dots,n$.

We parametrize elements of $U$ and $U^-$ as
\begin{equation}
	u_z = \begin{bmatrix} 1 & z \\ 0 & 1 \end{bmatrix} \in U
	\quad	\text{and}	\quad
	u_w^- = \begin{bmatrix} 1 & 0 \\ w & 1 \end{bmatrix} \in U^-,
\end{equation}
where $z \in \C^{n}=\Mat{n \times 1}(\C)$ (column vector) and $w \in (\C^{n})^{*}=\Mat{1 \times n}(\C)$ (row vector).

\section{Twisted Cotangent Bundle}

As we mentioned earlier, the $\Gc$-orbit $\ourorbit$ is holomorphically isomorphic to the twisted cotangent bundle of $\CP^n$, which is an affine bundle on $\CP^n$ and is denoted by $(T^* \CP^n)_\lambda$.
In this section, we first recall its construction (see \cite{Donaldson02_holomorphic_discs} \cite{equiv_sympl} for details), and then introduce a $G$-invariant function on $(T^* \CP^n)_\lambda$.

Let $\{ U_\alpha \}_{\alpha=0,1,\dots,n}$ be the standard open covering of $\CP^n$, i.e. 
\begin{align*}
	U_\alpha &= \big \{ [ \zeta_{0}:\zeta_{1}: \dots : \zeta_{n} ] \in \CP^n ; \zeta_{\alpha} \ne 0 \big \}  
\end{align*}
and $\varphi_\alpha : U_\alpha \xrightarrow{\sim} \C^n$ the holomorphic local chart induced from the homeomorphism $U_\alpha \simeq \sigma_\alpha U Q/Q$.
Introducing holomorphic coordinates $(z_\alpha,\xi_\alpha) = (z_{\alpha, 1}, \dots, z_{\alpha, n}, \xi_{\alpha, 1}, \dots, \xi_{\alpha, n} )$ on the trivial cotangent bundle $T^* U_\alpha$,
we denote a point of $T^* U_\alpha$ by 
\[
	\xi_\alpha\, \rmd z_\alpha := \sum_{i=1}^{n} \xi_{\alpha, i}\, \rmd z_{\alpha, i}.
\]
In our setting, it is natural to regard $z_{\alpha}$ as a column vector $\tp{(z_{\alpha,1},\dots,z_{\alpha,n})}$ and $\xi_{\alpha}$ as a row vector $(\xi_{\alpha,1},\dots,\xi_{\alpha,n})$.

Now we glue the cotangent bundles $\{ T^* U_\alpha \}_{\alpha=0,1,\dots,n}$ on the double intersection $U_\alpha \cap U_\beta$ by an equivalence relation:
\begin{equation}
	(T^* \CP^n)_\lambda := \bigsqcup_{\alpha = 0}^{n} T^* U_\alpha / \sim,
\end{equation}
where a point $\xi_\beta\, \rmd z_\beta$ in $T^* U_\beta|_{U_\alpha \cap U_\beta}$ is defined to be equivalent to 
a point $\xi_\alpha\, \rmd z_\alpha$ in $T^* U_\alpha|_{U_\alpha \cap U_\beta}$ if and only if the following two conditions hold:
\begin{align}
	z_\beta &= ( \varphi_\beta \circ \varphi_\alpha^{-1}) (z_\alpha),
		\label{e:z_alpha and z_beta}
	\\
	\xi_\beta\, \rmd z_\beta &= \xi_\alpha\,\rmd z_\alpha - \<\lambda,\rmd t_{\alpha \beta} {t_{\alpha \beta}}^{-1}\>.
		\label{e:xi_alpha and xi_beta}
\end{align}
Here $\<\cdot,\cdot\>$ denotes the pairing between $\g^*$ and $\g$, and $t_{\alpha \beta}$ is a unique element of $L$ determined by the relation  
\begin{equation}
\label{e:xi_alpha and xi_beta 2}
	\sigma_{\beta}\, u_{z_\beta} u^{-}_{w_\beta}. e_L = \sigma_{\alpha}\, u_{z_\alpha} u^{-}_{w_\alpha}. e_L,
\end{equation}
where $e_{L}$ denotes the origin of $\Gc / L$, and $w_\alpha,\, w_\beta \in (\C^n)^*$ are related to $\xi_\alpha, \, \xi_\beta \in (\C^n)^*$ by $\xi_{\alpha} = -s w_{\alpha},\,\xi_{\beta} = -s w_{\beta}$.
In fact, Eq.~\eqref{e:xi_alpha and xi_beta 2} implies a unique existence of $t_{\alpha \beta} \in L$ that satisfies 
$\sigma_{\beta} u_{z_{\beta}} u^{-}_{w_{\beta}} = \sigma_{\alpha} u_{z_{\alpha}} u^{-}_{w_{\alpha}} t_{\alpha \beta}$ under the condition \eqref{e:z_alpha and z_beta}.

We denote the canonical projection by $\pi_\lambda$:
\begin{equation}
	\pi_\lambda: (T^* \CP^n)_{\lambda} \to \CP^n.
\end{equation}
By definition, $\pi_\lambda^{-1}(U_\alpha) = T^* U_\alpha$, $\alpha=0,1,\dots, n$.

In what follows, we will denote the twisted cotangent bundle by $\ourZ$ for brevity:
\begin{equation}
\label{e:twd ctg bdle for short}
	\ourZ := (T^* \CP^n)_{\lambda}.
\end{equation}
Since the second term in the right-hand side of \eqref{e:xi_alpha and xi_beta} is exact, taking the exterior derivative of its both sides, one can equip $\ourZ$ with a holomorphic symplectic form $\omega_{+}$ of type $(2,0)$ given in terms of the local coordinates%
		\footnote{In what follows, we suppress the index $\alpha$ if it is clear in which coordinate patch $U_\alpha$ we are working.
		}
$(z,\xi)=(z_1,\dots,z_n,\xi_1,\dots,\xi_n)$ by
\begin{equation}
\label{e:local expression of omega_+}
	\omega_{+}|_{\pi_\lambda^{-1}(U_\alpha)} = \sum_{i=1}^{n} \rmd z_i \wedge \rmd \xi_i.
\end{equation}
Similarly, $\ourZ$ is equipped with an anti-holomorphic symplectic form $\omega_-$ on $\ourZ$ of type $(0,2)$ given by
\begin{equation}
	\omega_{-}|_{\pi_\lambda^{-1}(U_\alpha)} =\sum_{i=1}^{n} \rmd \bar z_i \wedge \rmd \bar \xi_i.
\end{equation}
We remark that there is a natural holomorphic $\Gc$-action $\gamma_{\!\lambda}$ on $\ourZ$ that covers the left translation on $\Gc / Q \simeq \CP^n$, i.e.~the following diagram commutes for all $g \in \Gc$: 
\begin{equation}
\begin{tikzcd}
	\ourZ \ar[r,"","\gamma_{\!\lambda}(g)"]	\ar[d,"\pi_\lambda"']	& \ourZ	\ar[d,"\pi_\lambda"] 	\\
	\CP^n \ar[r,"","g."]											& \CP^n.
\end{tikzcd}
\end{equation}
Note also that $\gamma_{\!\lambda}$ is Hamiltonian with respect to $\omega_+$, and that there is an anti-holomorphic Hamiltonian action with respect to $\omega_-$ as well.

Let $\mu_+$ and $\mu_-$ denote the moment maps $\ourZ \to \g^*$ with respect to $\omega_+$ and $\omega_-$ respectively.
Then $\mu_+$ provides a holomorphic $\Gc$-equivariant symplectic isomorphism from $\ourZ$ onto $\ourorbit$:
\begin{equation}
	\mu_+ : \ourZ \xrightarrow{\sim} \ourorbit,
\end{equation}
where the symplectic structure on $\ourorbit$ is given by the Kirillov-Kostant-Souriau form (\cite[Theorem 3.14]{equiv_sympl}).
Explicitly, $\mu_+$ is given on $\pi_\lambda^{-1}(U_\alpha)$ in terms of the local coordinates $(z,\xi)$ by
\begin{equation}
	\mu_+|_{\pi_\lambda^{-1}(U_{\alpha})} (z,\xi) = \Ad^*(g) \lambda \quad \text{with} \quad g = \sigma_\alpha\, u_z u^-_w,
\end{equation}
where $\xi=-s w$.

\begin{proposition}
If we denote the moment maps $\mu_+,\,\mu_-: \ourZ \to \g^*$ followed by the identification $\g^* \simeq \g$ via the trace form $B$, also by the same notations,
then we have
\begin{equation}
	\mu_- = - \mu_+^{*},
\end{equation}
where ${}^{*}$ denotes the transpose conjugate of matrices.
\end{proposition}
\begin{proof} 
The proof is straightforward, and is left to the reader.
\end{proof}

Now we are ready to introduce our main player.
Namely, set 
\begin{equation}
\label{e:seed}
\begin{aligned}
	\tau &:= - B(\mu_+, \mu_-) + \frac{\abs{s}^2}{n+1}
		\\
		&= B(\mu_+, \mu_+^*) + \frac{\abs{s}^2}{n+1},
\end{aligned}
\end{equation}
which is given in terms of the local coordinates $(z,\xi)$ by
\begin{equation}
\label{e:local expression of seed}
	\tau |_{\pi_\lambda^{-1}(U_\alpha)} = ( \norm{z}^2 + 1 )( \norm{\xi}^2 + \abs{s - \xi z}^2 ),
\end{equation}
where we regard $z \in \C^n$ (column vector) and $\xi \in (\C^n)^*$ (row vector), and $\norm{\,\cdot\,}^2$ denotes the norm squared, i.e.~$\norm{z}^2=z^* z, \norm{\xi}^2=\xi \xi^*$.
Note that we need the constant term in \eqref{e:seed} to factorize the local expression of $\tau$ as in \eqref{e:local expression of seed}.

\begin{remark}
By definition, $\tau$ is a $\Gu$-invariant function on $\ourZ$ with values in $\R_{\geqsl 0}$; it takes its minimum $\abs{s}^2$ on the submanifold $\ourLSmfd$, where $\ourLSmfd$ is locally given by
\begin{equation}
\label{e:Laglangian submfd}
	\ourLSmfd|_{\pi_\lambda^{-1}(U_\alpha)}=\left\{ \frac{s \, z^* \rmd z}{1+\norm{z}^2} ; z \in U_\alpha \right\} \subset T^* U_\alpha.
\end{equation}
Note that $\ourLSmfd$ provides an example of the \emph{LS-submanifold} of $\ourZ$, i.e.~it is Lagrangian with respect to $\im \omega_+$ and symplectic with respect to $\re \omega_+$ (see \cite{Donaldson02_holomorphic_discs}).
One also notes that $\ourLSmfd$ corresponds to the $\Gu$-orbit $\mathscr O_{\!\lambda}$ of $\lambda$ under the isomorphism $\mu_+: \ourZ \to \ourorbit$:
\begin{equation*}
\begin{tikzcd}
	\ourZ \ar[r,"\sim"',"\mu_+"]						& \ourorbit							 \\
	\ourLSmfd \ar[r,"\sim"',"\mu_+"]	\ar[u, hook]	& {\mathscr O}_{\!\lambda} \ar[u, hook],
\end{tikzcd}
\end{equation*}
and thus $\ourLSmfd \simeq {\mathscr O}_{\!\lambda} \simeq \Gu / (\Gu \cap L) = \SU(n+1) / S(\U(n) \times \U(1)) \simeq \CP^n$.
%
\end{remark}

\section{Ricci-flat K{\"a}hler Metric}

Henceforth, we assume that $s \ne 0$.

In this section, we find Ricci-flat K\"ahler metrics on the twisted cotangent bundle $Z$ under an Ansatz.
Namely, for an $\R$-valued function $f \in C^\infty(\R)$ in a single variable, let us consider a 1-form $\beta$ on $\ourZ$ given by
\begin{equation}
\label{e:one-form}
	\beta := \frac{1}4 f(\tau)\, \rmd^c \tau = \frac{\ai}4 f(\tau) (\antidel \tau - \del \tau)
\end{equation}
and define a 2-form $\omega_I$ by $\omega_I:=\rmd \beta$, where $\tau \in C^\infty(\ourZ)$ is the $\Gu$-invariant function given by \eqref{e:seed}.
Note then that
\begin{equation}
\label{e:omega_I}
	\omega_I = \frac{\ai}2 (f(\tau) \del \antidel \tau + f'(\tau) \del \tau \wedge \antidel \tau)
\end{equation}
is of type $(1,1)$, and satisfies
\begin{equation}
\label{e:omega & I}
	\omega_I (Iv,Iw)=\omega_I (v,w)  \quad (v,w \in \mathfrak X (\ourZ)),
\end{equation} 
where $I$ denotes the canonical complex structure on $\ourZ$.
Furthermore, one can see that $g^* \omega_I = \omega_I$ for all $g \in \Gu$ since $\tau$ is $\Gu$-invariant.

In what follows, we will be concerned with such functions $f \in C^\infty(\R)$ that satisfy
\begin{equation}
\label{e:positivity omega_I}
	\omega_I(v, I v) > 0
\end{equation}
for all nonzero $v \in \mathfrak X (\ourZ)$.
Let $g$ be the metric on $\ourZ$ whose fundamental $2$-form is $\omega_I$.
Namely, 
\begin{equation}
\label{e:def of metric g}
	g(v,w):=\omega_I ( v, I w) \quad (v,w \in \mathfrak{X}(\ourZ)).
\end{equation}
Then $(\ourZ,g,I)$ is a K\"ahler manifold if the condition \eqref{e:positivity omega_I} is satisfied.
In particular, $\omega_I$ is a symplectic form on $\ourZ$ (i.e.~non-degenerate on $\ourZ$).

\begin{remark}
\label{r:ansatz on beta}
It is clear from \eqref{e:xi_alpha and xi_beta} that if $\lambda$ tends to $0$ then our twisted cotangent bundle $\ourZ=(T^*\CP^n)_\lambda$ reduces to the cotangent bundle $T^*\CP^n$.
Kapustin and Saulina showed in \cite{KapustinSaulina09} that Eguchi-Hanson metric on $T^*\CP^1$, which is not only Ricci-flat but also hyperk\"ahler, is obtained if one takes $f(\tau)$ in \eqref{e:one-form} to be $(1+\tau)^{1/2}/\tau$.
This is the reason we assume the 1-form $\beta$ to be of the form given above.

\end{remark}

For the sake of space, we sometimes use the following convention about indices.
Namely, for $i=1,2,\dots,n$, we regard $\xi_i$ as the $(n+i)$-th coordinate:
\begin{equation}
\label{e:convention about idx}
	z_{n+i} := \xi_i.
\end{equation}
Under the convention \eqref{e:convention about idx}, let $g_{i \bar j}$, $i,j=1,2,\dots,2n$, denote (twice) the components of the metric $g$ in terms of local coordinates:
\begin{equation}
\label{e:local metric}
	g|_{\pi_\lambda^{-1}(U_\alpha)} = \frac12 \sum_{i,j=1}^{2n} g_{i \bar j}\, \rmd z_i \rmd \bar z_j
\end{equation}
and set 
\begin{equation}
\label{e:matrix corresponding to g}
	A_\alpha := (g_{i \bar j})_{i,j=1,2,\dots,2n} \in C^\infty(\pi_\lambda^{-1}(U_\alpha),\End(\C^{2n})).
\end{equation}

Recall that the Ricci form $\rho$ on the K\"ahler manifold $(\ourZ, g, I)$ with respect to the Levi-Civita connection is locally given by
\[
	\rho|_{\pi_\lambda^{-1}(U_\alpha)} = \ai\, \antidel \del \log \det A_\alpha. 
\]
The following theorem is a key to find Ricci-flat metrics. 
%
%
\begin{theorem}
\label{t:ricci-flat}
The metric $g$ satisfies $\det A_\alpha = 1$ if and only if the function $f$ satisfies that
\begin{equation}
\label{e:ricci-flat eq}
	2 f'(\tau) f(\tau)^{2n-1} \tau^n (\tau - \tau_0) + f(\tau)^{2n} \tau^{n-1} (2\tau - \tau_0) = 1,
\end{equation}
where we set $\tau_{0}=\abs{s}^2$.
\end{theorem}
\begin{proof}
By definition, the components $g_{i \bar j}$ are given by 
\begin{equation}
	g_{i \bar j} = f(\tau) \tau_{i, \bar j} + f'(\tau) \tau_{i} \tau_{\bar j}
\end{equation}
with $\tau_i:=\pd_{z_i} \tau, \, \tau_{\bar i}:=\pd_{\bar z_i} \tau$ and $\tau_{i, \bar j}:=\pd_{z_i} \pd_{\bar z_j} \tau$ for $i,j=1,2,\dots,2n$ under the convention \eqref{e:convention about idx}.
Setting $c:=f(\tau), c_1:=f'(\tau)$ for brevity and
\begin{align}
	\vec{a}_j &:= \tp{(\tau_{1, \bar j},\dots,\tau_{2n, \bar j})} \quad (j=1,2,\dots,2n), 
		\label{e:mat g by tau_xy}
	\\
	\vec{b} &:= \tp{(\tau_1,\dots,\tau_{2n})}, 
		\label{e:mat g by tau_x}
\end{align}
we write $A_\alpha$ as
\begin{equation}
\label{e:mat 4 metric g}
	A_\alpha 	= c ( \vec{a}_1, \vec{a}_2,\dots, \vec{a}_{2n} ) + c_1 ( \tau_{\bar 1} \vec{b}, \tau_{\bar 2} \vec{b},\dots, \tau_{\overline{2n}} \vec{b} ).
\end{equation}
Therefore,
\begin{equation}
	\det A_\alpha = c^{2n} \det(\vec{a}_1,\dots,\vec{a}_{2n}) 
				+ c^{2n-1} c_1 \sum_{k=1}^{2n}  \tau_{\bar k} \det (\vec{a}_1,\dots,\vec{a}_{k-1}, \hspace{-3pt}\overset{k\text{-th}}{\vec{b}}\hspace{-3pt}, \vec{a}_{k+1},\dots,\vec{a}_{2n}).
\end{equation}
Now the theorem follows from Propositions \ref{p:det mat tau_xy} and \ref{p:det mat tau_x tau_y} below.
\end{proof}

In the proof of the propositions below, we will use the following notations: 
\begin{equation}
\begin{aligned}
\label{e:psi1 and psi2}
	\psi_1 &:= \norm{z}^2 + 1, \\
	\psi_2 &:= \norm{\xi}^2 + \abs{s_{xz}}^2, \\
	s_{xz} &:= s - \xi z
\end{aligned}
\end{equation}
with $z=\tp{(z_1,z_2,\dots,z_n)} \in \C^n,\, \xi=(\xi_1,\xi_2,\dots,\xi_n) \in (\C^n)^*$.
Then, it is immediate from the local expression \eqref{e:local expression of seed} of $\tau$ to show that
\begin{equation}
\label{e:tau_x}
\begin{aligned}
	& \tau_{i}  = \psi_2\, \bar z_i -\psi_1 \bar s_{xz}\, \xi_i,
		\\
	& \tau_{n+i} = \psi_1 (\bar \xi_i - \bar s_{xz}\, z_i),
\end{aligned}
\end{equation}
for $i=1,\dots,n$, and
\begin{equation}
\label{e:tau_xy}
\begin{aligned}
	& \tau_{i, \bar j}	= \psi_2\, \delta_{ij} - s_{xz}\, \bar z_i \bar \xi_j - \bar s_{xz}\, \xi_i z_j + \psi_1\, \xi_i \bar \xi_j,
		\\
	& \tau_{i, \overline{n+j}} 	= \psi_1\, \xi_i \bar z_j + \bar z_i \xi_j - s_{xz}\, \bar z_i \bar z_j,
		\\
	& \tau_{n+i,\bar j}			= \psi_1\, z_i \bar \xi_j + \bar \xi_i z_j - \bar s_{xz}\, z_i z_j,
		\\
	& \tau_{n+i, \overline{n+j}} = \psi_1 ( \delta_{ij} + z_i \bar z_j ) 
\end{aligned}
\end{equation}
for $i,j=1,\dots,n$, where $\delta_{ij}$ denotes the Kronecker delta.

%
%
\begin{proposition}
\label{p:det mat tau_xy}
Let $\vec{a}_i$ $ $(i=1,2,\dots,2n$)$ be given by \eqref{e:mat g by tau_xy}.
Then one has
\begin{equation}
	\det(\vec{a}_1,\vec{a}_2,\dots,\vec{a}_{2n}) = \tau^{n-1} (2\tau - \tau_0), 
	\label{e:mat_cc}
\end{equation}
\end{proposition}
\begin{proof}
Let us write the matrix $(\vec{a}_1,\vec{a}_2,\dots,\vec{a}_{2n})$ in the block form: 
\begin{equation}
\label{e:block form tau_xy}
	(\vec{a}_1,\vec{a}_2,\dots,\vec{a}_{2n}) =: \begin{bmatrix} a_{11} & a_{12} \\ a_{21} & a_{22} \end{bmatrix}
\end{equation}
with $a_{ij} \in \Mat{n \times n}(\C)$ ($i,j=1,2$)%
	\footnote{To be precise, each block $a_{ij}$ is a $C^\infty$-function taking values in $\Mat{n \times n}(\C)$; 
				we will use this kind of abbreviation below.}%
.
Then, it follows from \eqref{e:tau_xy} that%
	\footnote{It would be more comprehensible to introduce bra- and ket-vectors and denote $\bar z\, \bar \xi$ by $|\bar z \> \<\bar \xi|$ etc.
				However, we will adopt a bit awkward notation to keep expressions simple.}
\begin{equation}
\label{e:submat cc}
\begin{aligned}
	a_{11} &= \psi_2 1_n - s_{xz}\, \bar z \, \bar \xi - \bar s_{xz} \tp{\xi} \tp{z} + \psi_1 \tp{\xi} \, \bar \xi,	\\
	a_{12} &= \psi_1 \tp{\xi} \, z^* + \bar z \, \xi - s_{xz} \, \bar z \, z^*, \\
	a_{21} &= \psi_1 z \, \bar \xi + \xi^* \tp{z} - \bar s_{xz} \, z \, \tp{z}, \\
	a_{22} &= \psi_1 (1_n + z \, z^*).
\end{aligned}
\end{equation}
Note in particular that
\begin{equation}
\label{e:hermiteness mat cc}
	{a_{11}}^*=a_{11}, \quad a_{21}={a_{12}}^* \quad \text{and} \quad {a_{22}}^*=a_{22},
\end{equation}
which also follows from the definition.

The matrix decomposition
\begin{equation}
\label{e:mat decomp UL}
\begin{aligned}
	\begin{bmatrix} a_{11} & a_{12} \\[1.5pt] a_{21} & a_{22} \end{bmatrix} 
		&= \begin{bmatrix} 1_n & a_{12}\, {a_{22}}^{-1}  \\ 0 & 1_n \end{bmatrix} 
				\begin{bmatrix} a_{11} - a_{12}\, {a_{22}}^{-1}\, a_{21} & 0 \\ a_{21} & a_{22} \end{bmatrix}
\end{aligned}
\end{equation}
implies that 
\[
	\det (\vec{a}_1,\vec{a}_2,\dots,\vec{a}_{2n})=\det(a_{11} - a_{12}\, {a_{22}}^{-1}\, a_{21}) \det a_{22}.
\]
Now, using \eqref{e:1/(1+rank-one mat)} in Lemma \ref{l:1/(1+rank-two mat)},
one obtains
\begin{equation}
\label{e:1/a_22}
	{a_{22}}^{-1} = {\psi_1}^{-1} \big( 1_n - {\psi_1}^{-1} z z^* \big),
\end{equation}
from which and \eqref{e:submat cc}, it follows that
\begin{align}
	a_{11} - a_{12} {a_{22}}^{-1} a_{21} 
		&= \psi_2 1_n + \tp{\xi} \bar \xi - (\psi_2 {\psi_1}^{-1} - \abs{s}^2 {\psi_1}^{-2}) \bar z \tp{z} - s {\psi_1}^{-1} \bar z \, \bar \xi - \bar s {\psi_1}^{-1} \tp{\xi} \tp{z}
			\notag	\\
		&= \psi_2 1_n + \tp{\xi} \big( \bar \xi - \bar s {\psi_1}^{-1} \tp{z} \big) + \bar z \big( m \tp{z} - s {\psi_1}^{-1} \bar \xi \big),
			\label{e:a_11 etc}
\end{align}
where we set $m:={\psi_1}^{-2} (\abs{s}^2 - \psi_1 \psi_2)$.
To apply Lemma \ref{l:det of 1+rank2mat} to \eqref{e:a_11 etc}, set
\begin{align*}
	Z_1 &= \tp{\xi}, & \tp{\,W_1} &= \bar \xi - \bar s {\psi_1}^{-1} \tp{z}, 
		\\
	Z_2 &=\bar z, 	& \tp{\,W_2} &= m \tp{z} - s {\psi_1}^{-1} \bar \xi.
\end{align*}
Then one finds that
\begin{align*}
	\< W_1,Z_1 \> &= \norm{\xi}^2 - \bar s {\psi_1}^{-1} \xi z,	& \< W_1,Z_2 \> &= \overline{\xi z} - \bar s {\psi_1}^{-1} \norm{z}^2,
		\\
	\< W_2,Z_1 \> &= m\, \xi z - s {\psi_1}^{-1} \norm{\xi}^2,	& \< W_2,Z_2 \> &= m \norm{z}^2 - s {\psi_1}^{-1} \overline{\xi z},
\end{align*}
and hence, 
\begin{align*}
	\trace \Lambda &= \norm{\xi}^2 + m \norm{z}^2 -\bar s {\psi_1}^{-1} \xi z - s {\psi_1}^{-1} \overline{\xi z},
		\\
	\det \Lambda &= {\psi_1}^{-1} \psi_2 ( \abs{\xi z}^2 - \norm{\xi}^2 \norm{z}^2 )
\end{align*}
in the notation of Lemma \ref{l:det of 1+rank2mat}, \eqref{e:2by2 mat}.
Thus, in view of \eqref{e:psi1 and psi2}, one has
\begin{align*}
	\det & (a_{11} -  a_{12}\, {a_{22}}^{-1}\, a_{21})	
			\\
		&= {\psi_2}^{n-2} \Big( {\psi_2}^2 + \psi_2 \big( \norm{\xi}^2 + m \norm{z}^2 - {\psi_1}^{-1} (\bar s\, \xi z + s\, \overline{\xi z}) \big) + \psi_2 {\psi_1}^{-1} (\abs{\xi z}^2 - \norm{\xi}^2 \norm{z}^2) \Big)
			\\
		&= {\psi_2}^{n-1} \Big( {\psi_2} + \psi_1^{-1} \big( \norm{\xi}^2 + \abs{s}^2 - (\bar s\, \xi z + s\, \overline{\xi z}) + \abs{\xi z}^2 \big) - \abs{s}^2 {\psi_1}^{-1} + m \norm{z}^2 \Big)
			\\
		&= {\psi_2}^{n-1} \bigg( \psi_2 \Big( 1 + \frac1{\psi_1}-\frac{\norm{z}^2}{\psi_1} \Big) + \frac{\abs{s}^2}{\psi_1} \Big(\frac{\norm{z}^2}{\psi_1} - 1 \Big) \bigg) 
			\\
		&= {\psi_1}^{-2} {\psi_2}^{n-1} ( 2 \psi_1 \psi_2 - \abs{s}^2 )
\end{align*}
and
\begin{equation*}
	\det a_{22} = {\psi_1}^n \det(1_n + z z^*) = {\psi_1}^{n+1}.
\end{equation*}
Therefore, one obtains that
\begin{align*}
	\det (\vec{a}_1,\vec{a}_2,\dots,\vec{a}_{2n}) &= {\psi_1}^{n-1} {\psi_2}^{n-1} ( 2 \psi_1 \psi_2 - \abs{s}^2 )
		\\
			&= \tau^{n-1} ( 2 \tau - \tau_0 ).
\end{align*}
This completes the proof.
\end{proof}

%
%
\begin{proposition}
\label{p:det mat tau_x tau_y}
Let $\vec{a}_i$ $ $(i=1,\dots,2n$)$ and $\vec{b}$ be given by \eqref{e:mat g by tau_xy} and \eqref{e:mat g by tau_x} respectively.
Then one has
\begin{equation}
\label{e:mat_cc1}
	\sum_{k=1}^{2n} \tau_{\bar k} \det( \vec{a}_1, \dots, \vec{a}_{k-1}, \hspace{-3pt}\overset{k\text{-th}}{\vec{b}}\hspace{-3pt}, \vec{a}_{k+1},\dots,\vec{a}_{2n} ) = 2 \tau^n (\tau - \tau_0).
\end{equation}
\end{proposition}
\begin{proof}
Set 
\(
	A'_\alpha := (\vec{a}_1,\dots,\vec{a}_{2n})
\)
and 
\[
	A'_{\alpha;k} (\vec{b}) :=( \vec{a}_1, \dots, \vec{a}_{k-1}, \hspace{-3pt}\overset{k\text{-th}}{\vec{b}}\hspace{-3pt},\vec{a}_{k+1},\dots,\vec{a}_{2n} )
	\qquad (k=1,\dots,2n)
\]
for brevity.
Then, the fact that 
\(
	\vec{x} = (1/\det A'_\alpha)\, {\vphantom{\Big( \Big)}}^t \!{\Big( \det A'_{\alpha;1}(\vec{b}), \dots, \det A'_{\alpha;2n}(\vec{b}) \Big)}
\) 
is the unique solution for $A'_\alpha \,\vec{x} = \vec{b}$ implies that 
\begin{align*}
	\sum_{k=1}^{2n} &\tau_{\bar k} \det( \vec{a}_1, \dots, \vec{a}_{k-1}, \vec{b},\vec{a}_{k+1},\dots,\vec{a}_{2n} )
	\\
	&= (\bar{\tau}_1,\bar{\tau}_2,\dots,\bar{\tau}_{2n}) {\vphantom{\Big( \Big)}}^t \!{\Big( \det A'_{\alpha;1}(\vec{b}), \dots, \det A'_{\alpha;2n}(\vec{b}) \Big)}
	\\
	&=(\det A'_\alpha)\, \vec{b}^* (A'_\alpha)^{-1} \vec{b},
\end{align*}
where we used the relation $\tau_{\bar k}=\bar{\tau}_k$.

The matrix decomposition
\begin{equation*}
\label{e:mat decomp UTU}
	A'_\alpha 
		= \begin{bmatrix} a_{11} & a_{12} \\[3pt] a_{21} & a_{22} \end{bmatrix} 
		= \begin{bmatrix} 1_n & a_{12}\, {a_{22}}^{-1}  \\[3pt]  0 & 1_n \end{bmatrix} 
				\begin{bmatrix} a_{11} - a_{12}\, {a_{22}}^{-1}\, a_{21} & 0 \\[3pt] 0 & a_{22} \end{bmatrix}
					\begin{bmatrix} 1_n & 0 \\[3pt] {a_{22}}^{-1}\, a_{21} & 1_n \end{bmatrix}
\end{equation*}
implies that 
\begin{equation}
\label{e:mat decomp UTU}
	(A'_\alpha)^{-1} 
		= \begin{bmatrix} 1_n & 0 \\[3pt] -{a_{22}}^{-1}\, a_{21} & 1_n \end{bmatrix} 
				\begin{bmatrix} (a_{11} - a_{12}\, {a_{22}}^{-1}\, a_{21})^{-1} & 0 \\[3pt] 0 & {a_{22}}^{-1} \end{bmatrix}
					\begin{bmatrix} 1_n & -a_{12}\, {a_{22}}^{-1}  \\[3pt]  0 & 1_n \end{bmatrix},
\end{equation}
and hence, using \eqref{e:hermiteness mat cc}, one finds
\begin{equation}
\label{e:b^* A' b}
\begin{aligned}
	\vec{b}^* (A'_\alpha)^{-1} \vec{b} 
		&= (b_1 - a_{12}\, {a_{22}}^{-1}\, b_2)^* (a_{11} - a_{12}\, {a_{22}}^{-1}\, a_{21})^{-1} (b_1 - a_{12}\, {a_{22}}^{-1}\, b_2) \\
		& \hspace{1.4em} + b_2^*\, {a_{22}}^{-1}\, b_2.
\end{aligned}
\end{equation}
Here we have written $\vec{b}$ into the block form as 
\begin{equation}
\label{e:submat vec b}
	\vec{b} =: \begin{bmatrix} b_1 \\[2pt] b_2 \end{bmatrix} \quad (b_1,b_2 \in \C^n).
\end{equation}
In view of \eqref{e:tau_x}, one finds that each block is given by
\begin{equation}
\label{e:b_1 and b_2}
\begin{aligned}
	b_1 &= \psi_2 \bar z - \psi_1 {\bar s}_{xz} \tp{\xi}, 
	\\
	b_2 &= \psi_1 ( \xi^* - {\bar s}_{xz} z ).
\end{aligned}
\end{equation}
Therefore, it follows from \eqref{e:1/a_22} and \eqref{e:b_1 and b_2} that the second term in the right-hand side of \eqref{e:b^* A' b} is given by
{\allowdisplaybreaks
\begin{align*}
	b_2^*\, {a_{22}}^{-1}\, b_2 
		&= \psi_1 (\xi - s_{xz} z^*) (1_n - {\psi_1}^{-1} z z^*) (\xi^* - \bar{s}_{xz} z)
			\\
		&= \psi_1 (\xi - s_{xz} z^*) (\xi^* - \bar s {\psi_1}^{-1} z)	
			\\
		&= \psi_1 \big(\norm{\xi}^2 - \bar s \xi z - s \overline{\xi z} + \abs{\xi z}^2 + \abs{s}^2  \norm{z}^2 {\psi_1}^{-1} \big)	
			\\
		&= \psi_1 \psi_2 - \abs{s}^2
			\\
		&= \tau - \tau_0.
\end{align*}
Similarly, 
\begin{align*}
	b_1 - a_{12}\, {a_{22}}^{-1}\, b_2 
		&= \psi_2 \bar{z} - \psi_1 \bar{s}_{xz} \tp{\xi} - (\tp{\xi} z^* + \bar{z} \xi - s {\psi_1}^{-1} \bar{z} z^*) (\xi^* - \bar{s}_{xz} z)
			\\
		&= \bar{s} {\psi_1}^{-1} ( s \bar{z} - \psi_1 \tp{\xi} ).
\end{align*}
}
Now, applying Lemma \ref{l:1/(1+rank-two mat)}, one finds that the inverse of the matrix \eqref{e:a_11 etc} is given by
\begin{equation*}
	(a_{11} - a_{12}\, {a_{22}}^{-1}\, a_{21})^{-1} 
		= {\psi_2}^{-1} \bigg( 
			1_n - \frac{\psi_1}{2 \psi_1 \psi_2 - \abs{s}^2} \Big( 
						\tp{\xi} \bar{\xi}\, -\, \bar{s}_{xz} \tp{\xi} \tp{z}\, -\, s_{xz}\, \bar{z}\, \bar{\xi} \,+\, (m \psi_1 \,-\,\norm{\xi}^2) \bar{z} \tp{z} 
															\Big)
						\bigg)
\end{equation*}
with $m= \psi_1^{-2}(\abs{s}^2 - \psi_1 \psi_2)$ as above.
Therefore, one sees that
\begin{align*}
	(a_{11} - & a_{12}\, {a_{22}}^{-1}\, a_{21})^{-1} ( b_1 - a_{12}\, {a_{22}}^{-1}\, b_2 )
			\\
		&= \frac{\bar s (\psi_1 \psi_2)^{-1}}{2 \psi_1 \psi_2 - \abs{s}^2} \bigg(
							(2 \psi_1 \psi_2 - \abs{s}^2 ) ( s \bar z - \psi_1 \tp{\xi}) 
							- \psi_1 (\bar \xi - \bar{s}_{xz} \tp{z}) (s \bar z - \psi_1 \tp{\xi}) \tp{\xi}
			\\
		&	\hspace{12.5em}
			- \psi_1 \big( -s_{xz} \bar \xi + (m \psi_1 - \norm{\xi}^2) \tp{z} \big) (s \bar z - \psi_1 \tp{\xi})\, \bar z\,
					\bigg)
			\\
		& =\frac{\bar s (\psi_1 \psi_2)^{-1}}{2 \psi_1 \psi_2 - \abs{s}^2} \bigg(
				  \Big\{ s(2\psi_1 \psi_2 - \abs{s}^2) - \psi_1 \big( -s_{xz} \bar\xi +(m \psi_1 - \norm{\xi}^2) \tp{z} \big) (s \bar z - \psi_1 \tp{\xi}) \Big\}\, \bar z
			\\
		&	\hspace{12.8em}
				 	- \Big\{ \psi_1(2 \psi_1 \psi_2 - \abs{s}^2) + \psi_1(\bar{\xi} - \bar{s}_{xz} \tp{z})(s \bar z - \psi_1 \tp{\xi} ) \Big\} \tp{\xi}\,
				\bigg)
			\\
		&= \frac{\bar s (\psi_1 \psi_2)^{-1}}{2 \psi_1 \psi_2 - \abs{s}^2} \big( {\psi_1}^2 \psi_2 (s - \xi z) \bar z - {\psi_1}^2 \psi_2 \tp{\xi} \big) 
			\\
		& = \frac{\bar{s} \psi_1}{2 \psi_1 \psi_2 - \abs{s}^2} (s_{xz} \bar z - \tp{\xi} ),
\end{align*}
and hence,
{\allowdisplaybreaks
\begin{align*}
	(b_1 - a_{12}\, {a_{22}}^{-1}\, & b_2)^*  (a_{11} - a_{12}\, {a_{22}}^{-1}\, a_{21})^{-1}  (b_1-a_{12}\, {a_{22}}^{-1}\, b_2)
		\\
		&= \frac{\abs{s}^2}{2 \psi_1 \psi_2 - \abs{s}^2} (\bar s \tp{z} - \psi_1 \bar \xi) (s_{xz} \bar z - \tp{\xi})
		\\
		&= \frac{\abs{s}^2}{2 \psi_1 \psi_2 - \abs{s}^2} \big( \psi_1 (\abs{s}^2 -\bar s \xi z -s \overline{\xi z} +\abs{\xi z}^2 + \norm{\xi}^2) - \abs{s}^2 \big)
		\\
		&=  \frac{\abs{s}^2}{2 \psi_1 \psi_2 - \abs{s}^2} (\psi_1 \psi_2 - \abs{s}^2)
		\\
		&= \frac{\tau_0}{2 \tau - \tau_0} (\tau - \tau_0).
\end{align*}
}
Thus, using the result of Proposition \ref{p:det mat tau_xy}, one obtains that
\begin{align*}
	(\det A'_\alpha)\, \vec{b}^* (A'_\alpha)^{-1} \vec{b} 
		&= \tau^{n-1} (2 \tau - \tau_0) \Big( 1+ \frac{\tau_0}{2 \tau - \tau_0} \Big) (\tau - \tau_0)
		\\
		&= 2 \tau^n (\tau - \tau_0).
\end{align*}
This completes the proof of the proposition, and hence of the theorem.
\end{proof}

One can easily solve the ODE given in \eqref{e:ricci-flat eq} to obtain a solution
\begin{equation}
\label{e:ricci-flat sol}
	f(\tau) = \left( \frac{\big( a + (\tau-\tau_0)^n \big)^{1/n}}{\tau (\tau - \tau_0)} \right)^{1/2}
\end{equation}
with $a$ a nonnegative constant.

Let us consider the solution $f(\tau)$ with $a=0$ in \eqref{e:ricci-flat sol}, i.e.~$f(\tau)=\tau^{-1/2}$.
Then one sees that the corresponding metric $g$ is positive-definite, and hence provides a Riemannian metric on $\ourZ$.
In fact, since $\ourZ$ is $\Gc$-homogeneous, it suffices to show that $g$ is positive-definite at a point of $\ourZ$, say $(z,\xi)=(0,0) \in \pi_\lambda^{-1}(U_\alpha) \subset \ourZ$.
Now, by \eqref{e:tau_x} and \eqref{e:tau_xy}, one has
\[
	A_\alpha|_{(z,\xi)=(0,0)} = \begin{bmatrix} \abs{s} 1_n & 0 \\[2pt] 0 & \abs{s}^{-1} 1_n \end{bmatrix},
\]
which is positive-definite.

\section{Hyperk\"ahler Metric}



In this section, we pick a hyperk\"ahler metric out of the Ricci-flat metric obtained in the previous section.
First, we recall the definition of a hyperk\"ahler manifold and a lemma due to Hitchin \cite{Hitchin87_monopole}:
\begin{definition}
Let $X$ be a $4n$-dimensional manifold.
Then a \emph{hyperk\"ahler structure} of $X$ consists of a Riemannian metric $g$ and a triple of almost complex structures $I,J,K$ satisfying the following conditions:
\begin{enumerate}{}{\setlength{\leftmargin}{10pt}}
\item
The metric $g$ is Hermitian with respect to all $I,J,K$:
\begin{equation}
\label{e:hermitian}
	g(Iv, Iw) = g(Jv, Jw) = g(Kv, Kw) = g(v,w)
\end{equation}
for $v,w \in {\mathfrak X}(X)$;

\item
The triple $(I,J,K)$ satisfies the quaternion identities: 
\begin{equation}
\label{e:quaternion}
	I^2 = J^2 = K^2 = IJK = -1.
\end{equation}
Moreover, the endomorphisms $I,J,K$ are covariant constant:
\begin{equation}
\label{e:covariant constant}
	\nabla I = \nabla J = \nabla K = 0,
\end{equation}
where $\nabla$ denotes the covariant derivative of the Levi-Civita connection of $g$.
\end{enumerate}

\end{definition}

Note that the condition \eqref{e:covariant constant} ensures the integrability of $I,J,K$, which is equivalent to their closedness of the associated 2-forms by the following lemma.

\begin{lemma}[\cite{Hitchin87_monopole}]
Let $(X,g)$ be a Riemannian manifold equipped with skew-adjoint endomorphisms $I,J,K$ of the tangent bundle $TX$ satisfying \eqref{e:hermitian} and \eqref{e:quaternion}.
Then $(X,g,I,J,K)$ is hyperk\"ahler if and only if the 2-forms $\omega_I, \omega_J, \omega_K$ that are determined by \eqref{e:2-forms} below are closed:
\begin{equation}
\label{e:2-forms}
	\omega_I(v,w) := g(Iv,w), \quad \omega_J(v,w) := g(Jv,w), \quad \omega_K(v,w) := g(Kv,w)  
\end{equation}
for $v,w \in {\mathfrak X}(X)$.
\end{lemma}

Now, we return to our case.
Let $\omega_J$ (resp. $\omega_K$) be the real part (resp. the imaginary part) of the holomorphic symplectic form $\omega_+$ whose local expression is given by \eqref{e:local expression of omega_+}:
\begin{equation}
	\omega_{+} = \omega_J + \ai\, \omega_K.
\end{equation}
Under the convention \eqref{e:convention about idx}, they are given by
\begin{align}
	\omega_J|_{\pi^{-1}_\lambda(U_\alpha)} &= \frac12 \sum_{i=1}^n ( \rmd z_i \wedge \rmd z_{n+i} + \rmd \bar z_i \wedge \rmd \bar z_{n+i} ),	\\
	\omega_K|_{\pi^{-1}_\lambda(U_\alpha)} &= -\frac{\ai}2 \sum_{i=1}^n ( \rmd z_i \wedge \rmd z_{n+i} - \rmd \bar z_i \wedge \rmd \bar z_{n+i} )
\end{align}
in terms of local coordinates.
Similarly, $\omega_I$ is locally given by 
\begin{equation}
	\omega_I|_{\pi^{-1}_\lambda(U_\alpha)} = \frac{\ai}2 \sum_{i,j=1}^{2n} g_{i \bar j}\, \rmd z_i \wedge \rmd \bar z_j.
\end{equation}

For a non-degenerate bilinear form $\omega$ on $T_p \ourZ$, $p \in \ourZ$, let $\omega^\flat$ denote the isomorphism from $T_p \ourZ$ onto $T_p^* \ourZ$ defined by $\omega^\flat(v)=\iota_v \omega$ with $\iota$ the interior product, i.e.
\[
	\omega^\flat(v)(w) =\omega (v,w) \quad (v,w \in T_p \ourZ).
\]
%
%
\begin{definition}
\label{d:cpx str J and K}
Let $J$ and $K$ be elements of $\Gamma(\End(T \ourZ))$ determined by $\omega_I,\,\omega_J$ and $\omega_K$ as follows:
\begin{equation}
	J = (\omega^\flat_I)^{-1} \circ \omega^\flat_K, \quad 
	K = (\omega^\flat_J)^{-1} \circ \omega^\flat_I.
\end{equation}
Furthermore, we also denote the canonical extensions of $I,J,K \in \Gamma(\End(T \ourZ))$ to $\Gamma(\End(\Tc \ourZ))$ by the same letters respectively, where $\Tc \ourZ := T \ourZ \otimes \C$.
\end{definition}

It is immediate to show that if $I^2 = J^2 = K^2 = IJK = -1$, then
\begin{equation}
\label{e:omega_J and omega_K}
	\omega_J(v,w) = g( Jv, w ), \quad \omega_K( v,w ) = g( Kv, w ),
\end{equation}
and
\begin{equation}
\label{e:hermitiancy of g wrt J and K}
	g(Jv, Jw) = g(Kv, Kw) = g(v,w)
\end{equation}
for $v,w \in {\mathfrak X}(\ourZ)$.
Needless to say, one has
\begin{equation}
\label{e:hermitiancy of g wrt I}
	g(Iv, Iw) = g(v,w) 	\quad ( v,w \in {\mathfrak X}(\ourZ) )
\end{equation}
by definition.

Under the convention \eqref{e:convention about idx}, let us introduce a basis $\Sigma$ for $\Tc \ourZ|_{\pi_\lambda^{-1}(U_\alpha)}$ given by
\[
	\Sigma = \big\{ \del_1, \dots, \del_{2n}, \antidel_{1},\dots, \antidel_{2n} \big\},
\]
where $\del_i := \del_{z_i}, \antidel_{i} := \del_{\bar z_i}$ $(i=1,\dots,2n)$.

%
%
\begin{proposition}
\label{p:J^2 and K^2}
Let $A_\alpha=(g_{i \bar j})_{i,j=1,\dots,2n}$ given by \eqref{e:matrix corresponding to g} with $f(\tau)$ general, and $J,K \in \Gamma(\End(\Tc \ourZ))$ be as in Definition \ref{d:cpx str J and K}.
Then one has $J^2 = K^2 = -1$ if and only if 
\begin{equation}
\label{e:condition Sp2n}
	\tp{A_\alpha} S A_\alpha = S, 
\end{equation}
where $S=\Big[ \begin{smallmatrix} 0 & 1_n \\[1.8pt] -1_n & 0\end{smallmatrix} \Big]$.
\end{proposition}
\begin{proof}
Denote by $\mathbb J$ and $\mathbb K$ the $4n \times 4n$ matrices that correspond to $J$ and $K$ with respect to the basis $\Sigma$ for $\Tc \ourZ|_{\pi_\lambda^{-1}(U_\alpha)}$.
Then one can show that they are given by
\begin{subequations}
\label{e:matrices J and K}
\begin{align}
	\mathbb J &= \begin{bmatrix} 0 & \tp{A^{-1}_\alpha} S^{-1} \\ A^{-1}_\alpha S^{-1} & 0 \end{bmatrix},
		\label{e:matrix J}
		\\
	\mathbb K &= \begin{bmatrix} 0 & - \ai S A_\alpha \\ \ai S \tp{A_\alpha} & 0 \end{bmatrix},
		\label{e:matrix K}
\end{align}
\end{subequations}
from which it follows immediately that ${\mathbb J}^2={\mathbb K}^2=-1$ if and only if $\tp{A_\alpha} S A_\alpha=S$.

Now, we will prove \eqref{e:matrices J and K}.
If one wrtites $v,\,J v \in \Tc_p \ourZ$ as
{\allowdisplaybreaks
\begin{align*}
	v	&=\sum_{i=1}^{2n} \epsilon_i \, \del_i + \sum_{i=1}^{2n} \epsilon^-_i \, \antidel_i 
		= (\del_1,\dots,\del_{2n},\antidel_1,\dots, \antidel_{2n}) 
			\begin{bmatrix} \vec{\epsilon}^+ \\ \vec{\epsilon}^- \end{bmatrix},
			\\
	Jv	&=\sum_{i=1}^{2n} \eta_i \, \del_i + \sum_{i=1}^{2n} \eta^-_i \, \antidel_i 
		= (\del_1,\dots,\del_{2n},\antidel_1,\dots, \antidel_{2n}) 
			\begin{bmatrix} \vec{\eta}^+ \\ \vec{\eta}^- \end{bmatrix}
\end{align*}
with $\vec{\epsilon}^+=\tp{(\epsilon_1,\dots,\epsilon_{2n})}, \vec{\epsilon}^-=\tp{(\epsilon^-_1,\dots,\epsilon^-_{2n})},\, \vec{\eta}^+=\tp{(\eta_1,\dots,\eta_{2n})}, \vec{\eta}^-=\tp{(\eta^-_1,\dots,\eta^-_{2n})}$, 
then the definition that $J v = (\omega^\flat_I)^{-1} \circ \omega^\flat_K (v)$, i.e.~$\omega^\flat_I(J v) = \omega^\flat_K (v)$ ($v \in \Tc_p \ourZ$) reads
\begin{align*}
	\epsilon_i &= \sum_{j=1}^{2n} g_{n+i,\bar j} \, \eta^-_j, & \epsilon_{n+i} &= - \sum_{j=1}^{2n} g_{i,\bar j} \, \eta^-_j,
		\\
	\epsilon^-_i &= \sum_{j=1}^{2n} g_{j, \overline{n+i}} \, \eta_j, & \epsilon^-_{n+i} &= - \sum_{j=1}^{2n} g_{j,\bar i} \, \eta_j
\end{align*}
}
for $i=1,\dots, n$, which can be written as
\begin{equation*}
	\vec{\epsilon}^+= S A_\alpha \, \vec{\eta}^-, \quad \vec{\epsilon}^-= S \tp{A_\alpha} \, \vec{\eta}^+.
\end{equation*}
Therefore, one has
\begin{equation*}
	\vec{\eta}^+= \tp{A_\alpha}^{-1} S^{-1} \vec{\epsilon}^-, \quad \vec{\eta}^-= A_\alpha^{-1} S^{-1} \vec{\epsilon}^+.
\end{equation*}
Namely,
\begin{equation*}
	\begin{bmatrix} \vec{\eta}^+ \\ \vec{\eta}^-\end{bmatrix}
		= \begin{bmatrix} 0 & \tp{A_\alpha}^{-1} S^{-1} \\ A_\alpha^{-1} S^{-1} & 0 \end{bmatrix}
			\begin{bmatrix} \vec{\epsilon}^+ \\ \vec{\epsilon}^- \end{bmatrix}.
\end{equation*}
Hence one obtains \eqref{e:matrix J}.

Similarly, writing $K v$ as
\[
	K v=\sum_{i=1}^{2n} \eta_i \, \del_i + \sum_{i=1}^{2n} \eta^-_i \, \antidel_i 
		= (\del_1,\dots,\del_{2n},\antidel_1,\dots, \antidel_{2n}) 
			\begin{bmatrix} \vec{\eta}^+ \\ \vec{\eta}^- \end{bmatrix},
\]
one finds that $\omega^\flat_J(K v) = \omega^\flat_I (v)$ reads
\begin{align*}
	\eta_i &= -\ai \sum_{j=1}^{2n} g_{n+i,\bar j} \, \epsilon_j, & \eta_{n+i} &= \ai \sum_{j=1}^{2n} g_{i,\bar j} \, \epsilon_j,
		\\
	\eta^-_i &= \ai \sum_{j=1}^{2n} g_{j, \overline{n+i}} \, \epsilon_j, & \eta^-_{n+i} &= - \ai \sum_{j=1}^{2n} g_{j,\bar i} \, \epsilon_j
\end{align*}
for $i=1,\dots, n$, which can be written as
\begin{equation*}
	\vec{\eta}^+= -\ai S A_\alpha \, \vec{\epsilon}^-, \quad \vec{\eta}^-= \ai S\, \tp{A_\alpha} \, \vec{\epsilon}^+.
\end{equation*}
Namely,
\begin{equation*}
	\begin{bmatrix} \vec{\eta}^+ \\ \vec{\eta}^- \end{bmatrix}
		= \ai \begin{bmatrix} 0 & -S A_\alpha \\ S\, \tp{A_\alpha} & 0 \end{bmatrix}
				\begin{bmatrix} \vec{\epsilon}^+ \\ \vec{\epsilon}^- \end{bmatrix}.
\end{equation*}
Hence one obtains \eqref{e:matrix K}.
\end{proof}

\begin{remark}
\label{r:cpx str I, Sp & E-H}
%
(i)\;
It is easy to verify that the canonical complex structure $I$ on $\ourZ$ satisfies $I = (\omega^\flat_K)^{-1} \circ \omega^\flat_J$, and that the matrix $\mathbb I$ corresponding to $I$ with respect to the basis $\Sigma$
for $\Tc \ourZ|_{\pi_\lambda^{-1}(U_\alpha)}$ is given by
\begin{equation}
\label{e:matrix I}
	\mathbb I= \ai \begin{bmatrix} 1_{2n} & 0 \\  0 & -1_{2n} \end{bmatrix}.
\end{equation}
It follows from \eqref{e:matrices J and K} and \eqref{e:matrix I} that ${\mathbb I}\,{\mathbb J}\,{\mathbb K}=-1$, hence $IJK=-1$, regardless of whether $A_\alpha$ satisfies \eqref{e:condition Sp2n} or not. 
%

%
(ii)\;
Eq.~\eqref{e:condition Sp2n} is nothing but the condition that $A_\alpha \in \Sp(2n,\C)$, \emph{the complex symplectic group}, where
\begin{equation*}
\Sp(2n,\C) = \big\{ g \in \GL{2n}(\C); \tp{g} S g = S \big\}
\end{equation*}
with \( S = \left[ \begin{smallmatrix} 0 & 1_n \\[2pt] -1_n & 0 \end{smallmatrix} \right] \).
Therefore, when $n=1$, all solutions given in \eqref{e:ricci-flat sol} with $a \geqsl 0$ provide hyperk{\"a}hler metrics on $\ourZ$ outside the Lagrangian submanifold $\ourLSmfd$ given by \eqref{e:Laglangian submfd}, 
since $\Sp(2,\C)=\SL{2}(\C)$. 
Note that if $\lambda$ tends to $0$ (i.e.~$s \to 0$), then the twisted cotangent bundle $\ourZ=(T^*\CP^1)_\lambda$ reduces to the cotangent bundle $T^*\CP^1$ (as we remarked in Remark \ref{r:ansatz on beta} above),
and the corresponding metric with $a>0$ coincides with the Eguchi-Hanson metric (\cite{EguchiHanson79}; see also \cite{KapustinSaulina09} for more details).

\end{remark}



Recall that our symplectic form $\omega_I$ is given by $\omega_I=\rmd \beta$ with 
\(
	\beta =  \frac{1}4 f(\tau)\, \rmd^c \tau 
\) and that the metric $g$ defined by
\begin{equation*}
	g(v,w) = \omega_I (v, I w)	\quad ( v,w \in {\mathfrak X}(\ourZ) )
\end{equation*} 
is Ricci-flat if and only if $f$ is of the form
\begin{equation}
\label{e:cicci-flat sol}
	f(\tau) = \left( \frac{\big( a + (\tau-\tau_0)^n \big)^{1/n}}{\tau (\tau - \tau_0)} \right)^{1/2}
\end{equation}
with $a$ a nonnegative constant, where $\tau_0=\abs{s}^2$.

\begin{theorem}
\label{t:hk_sol}
The metric $g$ defined by \eqref{e:def of metric g} provides a hyperk{\"a}hler metric on $\ourZ$ if and only if $a=0$, i.e.~if and only if $f(\tau)$ is given by
\begin{equation}
\label{e:hk_sol}
	f(\tau) = \tau^{-1/2}.
\end{equation}
 
\end{theorem}
\begin{proof}
By Proposition \ref{p:J^2 and K^2} and Remark \ref{r:cpx str I, Sp & E-H} (i), it suffices to show that $\tp{A_\alpha} S A_\alpha = S$ if and only if $a=0$.
Let us rewrite \eqref{e:mat 4 metric g} as
\begin{equation}
	A_\alpha = c A'_\alpha + c_1 A''_\alpha
\end{equation}
with $c=f(\tau),\, c_1=f'(\tau)$, and
\begin{align*}
	A'_\alpha &= (\tau_{i,\bar j})_{i,j=1, \dots, 2n} = (\vec{a}_1, \vec{a}_2,\dots, \vec{a}_{2n}), \\
	A''_\alpha &= (\tau_i \tau_{\bar j})_{i,j=1,\dots, 2n} = ( \tau_{\bar 1} \vec{b}, \tau_{\bar 2} \vec{b}, \dots, \tau_{\overline{2n}}\, \vec{b} ),
\end{align*}
where $\vec{a}_i, \vec{b}$ are given by \eqref{e:mat g by tau_xy} and \eqref{e:mat g by tau_x} respectively.
Since $A''_\alpha=\vec{b} {\vec{b}}^*$ and $\tp{\vec{b}} S \vec{b} = 0$, one has
\begin{equation}
	\tp{A''_\alpha}\, S\, A''_\alpha = 0.
\end{equation}
Therefore, one sees that 
{\allowdisplaybreaks
\begin{align}
\label{e:Sp condition}
	\tp{A_\alpha} S A_\alpha 
		&= c^2 \,\tp{A'_\alpha} S A'_\alpha + c c_1 ( \tp{A'_\alpha} S A''_\alpha + \tp{A''_\alpha} S A'_\alpha)
		\notag
		\\
		&= c^2 \,\tp{A'_\alpha} S A'_\alpha + c c_1 \Big( \tp{A'_\alpha} S A''_\alpha -\tp{ ( \tp{A''_\alpha} S A'_\alpha ) } \Big)
\end{align}
}
since $\tp{A''_\alpha} S A'_\alpha=-\tp{(\tp{A'_\alpha} S A''_\alpha)}$.
Using the notations in the proof of Propositions \ref{p:det mat tau_xy} and \ref{p:det mat tau_x tau_y}, i.e.
\begin{equation*}
	A'_\alpha=\begin{bmatrix} a_{11} & a_{12} \\ a_{21} & a_{22} \end{bmatrix},
		\qquad
	A''_\alpha=\begin{bmatrix} b_1 {b_1}^* & b_1 {b_2}^* \\[2pt] b_2 {b_1}^* & b_2 {b_2}^* \end{bmatrix}
\end{equation*}
with $a_{ij} \in \Mat{n \times n}(\C), b_i \in \C^n$ ($i,j=1,2$), 
one sees 
{\allowdisplaybreaks
\begin{align*}
	\tp{A'_\alpha} S A'_\alpha 
		&= 	\begin{bmatrix} 
				\tp{a_{11}} a_{21} - \tp{a_{21}} a_{11} & \tp{a_{11}} a_{22} - \tp{a_{21}} a_{12} \\[3pt] 
				\tp{a_{12}} a_{21} - \tp{a_{22}} a_{11} & \tp{a_{12}} a_{22} - \tp{a_{21}} a_{12} 
			\end{bmatrix},
		\\
	\tp{A'_\alpha} S A''_\alpha 
		&= 	\begin{bmatrix} 
				( \tp{a_{11}} b_{2} - \tp{a_{21}} b_1 ) {b_1}^* &  ( \tp{a_{11}} b_2 - \tp{a_{21}} b_1 ) {b_2}^* \\[3pt] 
				( \tp{a_{12}} b_{2} - \tp{a_{22}} b_1 ) {b_1}^* &  ( \tp{a_{12}} b_2 - \tp{a_{22}} b_1 ) {b_2}^* 
			\end{bmatrix}.	
\end{align*}
}
Recall from \eqref{e:submat cc} and \eqref{e:b_1 and b_2} that
{\allowdisplaybreaks
\begin{align*}
	a_{11} &= \psi_2 1_n - s_{xz}\, \bar z \, \bar \xi - \bar s_{xz} \tp{\xi} \tp{z} + \psi_1 \tp{\xi} \, \bar \xi,	
		&
	a_{12} &= \psi_1 \tp{\xi} \, z^* + \bar z \, \xi - s_{xz} \, \bar z \, z^*, \\
	a_{21} &= \psi_1 z \, \bar \xi + \xi^* \tp{z} - \bar s_{xz} \, z \, \tp{z}, 
		&
	a_{22} &= \psi_1 (1_n + z \, z^*),	\\
	b_1 &= \psi_2 \bar z - \psi_1 {\bar s}_{xz} \tp{\xi}, 
		&
	b_2 &= \psi_1 ( \xi^* - {\bar s}_{xz} z ).
\end{align*}
}
Thus, simple matrix calculations show that the $(1,1)$-, $(1,2)$-, and $(2,2)$-blocks of $\tp{A'} S A'$ are given by
\begin{subequations}
\label{e:A'SA'}
\renewcommand{\theequation}{\theparentequation}
{\allowdisplaybreaks
\begin{align}
	\tp{a_{11}} a_{21} - \tp{a_{21}} a_{11} &= (s_{xz} \overline{\xi z} - \norm{\xi}^2) (z \bar \xi - \xi^* \tp{z} ),
		\tag{\theequation$_{[1,1]}$} 
		\label{e:(1,1)-A'SA'}
		\\
	\tp{a_{11}} a_{22} - \tp{a_{21}} a_{12} &= \psi_1 \psi_2 1_n - \psi_1 s_{xz}\, \xi^* z^* + \bar s s_{xz}\, z z^* + \psi_1\, \xi^* \xi - \bar s\, z \xi,
		\tag{\theequation$_{[1,2]}$} 
		\label{e:(1,2)-A'SA'}
		\\
	\tp{a_{12}} a_{22} - \tp{a_{21}} a_{12} &= 0,
		\tag{\theequation$_{[2,2]}$} 
		\label{e:(2,2)-A'SA'}
		\\
\intertext{respectively. The $(2,1)$-block of $\tp{A'} S A'$ is equal to the negative transpose of its $(1,2)$-block:}
	\tp{a_{12}} a_{21} - \tp{a_{22}} a_{11} &= -\tp{(\tp{a_{11}} a_{22} - \tp{a_{21}} a_{12})}.
		\tag{\theequation$_{[2,1]}$} 
\end{align}
}
\end{subequations}

Similarly, one obtains that the $(1,1)$-, $(1,2)$-, $(2,1)$-, and $(2,2)$-blocks of $\tp{A'} S A''$ are given by
{\allowdisplaybreaks
\begin{subequations}
\label{e:A'SA''}
\renewcommand{\theequation}{\theparentequation}
\begin{align}
	( \tp{a_{11}} b_{2} - \tp{a_{21}} b_1 ) {b_1}^* 
		&= - s_{xz} {\psi_1}^2 (\psi_2 - s_{xz} \overline{\xi z} + \norm{\xi}^2)\,\xi^* \bar \xi - \bar s {\psi_2}^2\, z \tp{z}
			\notag \\
		& \hspace{2.65em} + \psi_1 \psi_2 (\psi_2 - s_{xz} \overline{\xi z} + \norm{\xi}^2)\, \xi^* \tp{z} + \bar s s_{xz }\psi_1 \psi_2\, z \bar \xi, 
		\tag{\theequation$_{[1,1]}$} 
			\label{e:(1,1)-A'SA''}
			\\
	( \tp{a_{11}} b_2 - \tp{a_{21}} b_1 ) {b_2}^* 
		&= \psi_1 \Big( \psi_1 (\psi_2 - s_{xz} \overline{\xi z} + \norm{\xi}^2 )\, \xi^* \xi + \bar s _{xz} \psi_2\, z z^* 
			\notag \\
		& \hspace{3.2em} - \bar s \psi_2\, z \xi - s_{xz} \psi_1 (\psi_2 - s_{xz} \overline{\xi z} + \norm{\xi}^2)\, \xi^* z^* \Big),
		\tag{\theequation$_{[1,2]}$} 
			\label{e:(1,2)-A'SA''}
			\\
	( \tp{a_{12}} b_2 - \tp{a_{22}} b_1 ) {b_1}^* 
		&= \bar s \psi_1 ( -s_{xz} \psi_1\, \tp{\xi} \bar \xi - s_{xz} \psi_2\, \bar z \tp{z} + \psi_2\, \tp{\xi} \tp{z}+ {s_{xz}}^2 \psi_1\, \bar z \bar \xi),
		\tag{\theequation$_{[2,1]}$} 
			\label{e:(2,1)-A'SA''}
			\\
	( \tp{a_{12}} b_2 - \tp{a_{22}} b_1 ) {b_2}^* 
		&= \bar s {\psi_1}^2 ( \tp{\xi} \xi + {s_{xz}}^2 \bar z z^* - s_{xz} \bar z \xi - s_{xz} \tp{\xi} z^* ),
		\tag{\theequation$_{[2,2]}$} 
			\label{e:(2,2)-A'SA''}
\end{align}
\end{subequations}
}%
respectively.
Thus, it is immediate from \eqref{e:(2,2)-A'SA'} and \eqref{e:(2,2)-A'SA''} to see that the $(2,2)$-block of \eqref{e:Sp condition} vanishes identically.

Now, let us show that the right-hand side of \eqref{e:Sp condition} equals $S=\Big[ \begin{smallmatrix} 0 & 1_n \\[2pt] -1_n & 0 \end{smallmatrix} \Big]$ blockwise if and only if $a=0$.

Assume first that $a=0$.
Then, noting that $c^2=\tau^{-1}$ and $c c_1 = -(1/2) \tau^{-2}$, one finds that the $(1,1)$-block of \eqref{e:Sp condition} equals
\begin{equation}
\label{e:(1,1)-block}
\begin{aligned}
	&\frac12 (\psi_1 \psi_2)^{-1} 
				\Big( 2 (s_{xz} \overline{\xi z} - \norm{\xi}^2) + \psi_2 - s_{xz} \overline{\xi z} + \norm{\xi}^2 - \bar s s_{xz} \Big) (z \bar \xi - \xi^* \tp{z})
	 \\
	=\,& \frac12 (\psi_1 \psi_2)^{-1} \Big( \psi_2 - \norm{\xi}^2 - \bar s s_{xz} + s_{xz} \overline{\xi z} \Big) (z \bar \xi - \xi^* \tp{z})
	 \\
	=\,& \frac12 (\psi_1 \psi_2)^{-1} \Big( \psi_2 - \norm{\xi}^2 - \abs{s_{xz}}^2 \Big) (z \bar \xi - \xi^* \tp{z}),
\end{aligned}
\end{equation}
which vanishes identically by \eqref{e:psi1 and psi2}.

Similarly, one finds that the $(1,2)$-block of \eqref{e:Sp condition} equals
\begin{align*}
	& (\psi_1 \psi_2)^{-1} \Big( \psi_1 \psi_2 1_n + \psi_1 \xi^* (\xi - s_{xz} z^*) - \bar s z (\xi - s_{xz} z^*) \Big)
	\\
	&\hspace{0.5em}  + \frac12 (\psi_1 \psi_2)^{-2} \Big( 
		  -{\psi_1}^2 ( \psi_2 - s_{xz} \overline{\xi z} + \norm{\xi}^2)\, \xi^* (\xi - s_{xz} z^*) + \bar s \psi_1 \psi_2\, z (\xi - s_{xz} z^*) 
	\\
	& \hspace{6.85em} + \bar s \psi_1 \big( -s_{xz} \psi_1\, \xi^* (\xi - s_{xz} z^*) + \psi_2\, z (\xi - s_{xz} z^*) \big)
			\Big)
	\\
	&\,= 1_n + \frac12 {\psi_2}^{-2} \big( \psi_2 - \bar s s_{xz} + s_{xz} \overline{\xi z} - \norm{\xi}^2 \big)\, \xi^* (\xi - s_{xz} z^*)
	\\
	&\,= 1_n + \frac12 {\psi_2}^{-2} \big( \psi_2 - \abs{s_{xz}}^2 - \norm{\xi}^2 \big)\, \xi^* (\xi - s_{xz} z^*)
	\\
	&\,= 1_n
\end{align*}
identically, again by \eqref{e:psi1 and psi2}.
Note that the $(2,1)$-block of the right-hand side of \eqref{e:Sp condition} is equal to the negative transpose of its $(1,2)$-block, and thus equals $-1_n$ identically.

Conversely, assume that $a \ne 0$ in \eqref{e:cicci-flat sol}.
Then, one can easily verify that the $(1,1)$-block of \eqref{e:Sp condition} does not vanish identically.
Hence $A_\alpha \notin \Sp(2n,\C)$.
This completes the proof.
\end{proof}

Note that $\omega_I = \frac12 \rmd \rmd^c\, \tau^{1/2}$ if $f(\tau)=\tau^{-1/2}$.
Namely, a K{\"a}hler potential $\Phi$ for $(\ourZ,g,I)$ is given by $\Phi=\tau^{1/2}$; its local expression is given by
\begin{equation}
	\Phi|_{\pi_\lambda^{-1}(U_\alpha)} = ( \norm{z}^2 + 1 )^{1/2} ( \norm{\xi}^{2} + \abs{s - \xi z}^{2} )^{1/2}
\end{equation}
in terms of the local coordinates $(z,\xi)$.



%
%

\appendix

\section{
}
\label{s:det rank2 etc}

Throughout the appendix, let $A$ denote a matrix given by
\begin{equation}
\label{e:rank2 mat}
	A = Z_1 \tp{\,W_1} + Z_2 \tp{\,W_2} \in \Mat{n \times n}(\C)
\end{equation}
for $Z_i,W_i \in \C^n$ ($i=1,2$).
Note that the rank of $A$ is at most 2 if $n \geqsl 2$.

For the matrix $A$ given in \eqref{e:rank2 mat}, let us introduce a $2 \times 2$ matrix $\Lambda$ by
\begin{equation}
\label{e:2by2 mat}
	\Lambda = \begin{bmatrix} 
				\< W_1, Z_1 \>  & \< W_1, Z_2 \> \\[3pt] 
				\< W_2, Z_1 \>  & \< W_2, Z_2 \>
			  \end{bmatrix} \in \Mat{2 \times 2}(\C),
\end{equation}
where $\<\cdot,\cdot\>$ denotes the canonical bilinear form on $\C^n$.

\begin{lemma}
\label{l:det of 1+rank2mat}
For $u \in \C$, one has
\begin{equation}
\label{e:det of 1+rank2mat}
	\det( u 1_n + A) = u^{n-2} ( u^2 + u \trace{\Lambda} + \det \Lambda ).
\end{equation}
\end{lemma}
\begin{proof}
First, note that
\begin{equation*}
	\det( u 1_n + A) = u^{n-2} \big( u^2 + u \trace{A} + \tfrac12( (\trace A)^2 - \trace{(A^2)}) \big)
\end{equation*}
since $A$ has its eigenvalues $\lambda_1,\lambda_2, \lambda_3 = \lambda_4 = \dots = \lambda_n=0$.
It is immediate to show that
\begin{equation*}
\label{e:trace Lambda & det Lambda}
\begin{aligned}
	\trace{\Lambda} &= \trace{A}, \\
	\det \Lambda &= \tfrac12 \big( (\trace A)^2 - \trace{(A^2)} \big)
\end{aligned}
\end{equation*}
if one notes $Z_i \tp{\,W_j} \, Z_k \tp{\,W_l} = \<W_j,Z_k\>\, Z_i \tp{\,W_l}$.
This completes the proof.
\end{proof}

Note in particular that one obtains $\det(1_n + A) = \det(1_2 + \Lambda)$ setting $u=1$ in \eqref{e:det of 1+rank2mat}.

\begin{lemma}
\label{l:1/(1+rank-two mat)}
Let $A \in \Mat{n \times n}(\C)$ and $\Lambda \in \Mat{2 \times 2}(\C)$ be matrices given by \eqref{e:rank2 mat} and \eqref{e:2by2 mat} respectively.
Then one has
\begin{equation}
\label{e:1/(1+rank-two mat)}
	(1_n + A)^{-1} = 1_n - \frac{1}{\det(1_2 + \Lambda)}
						\Big( (1+\lambda_{22})\, Z_1 \! \tp{\,W_1} - \lambda_{12}\, Z_1 \! \tp{\,W_2} - \lambda_{21}\, Z_2 \! \tp{\,W_1} + (1+\lambda_{11})\, Z_2 \! \tp{\,W_2} \Big)
\end{equation}
if $\det(1_2 + \Lambda) \ne 0$, where we set $\lambda_{ij}:=\<W_i, Z_j \>$ ($i,j=1,2$).

In particular, setting $Z_2=0$ or $W_2=0$, one has
\begin{equation}
\label{e:1/(1+rank-one mat)}
	(1_n + Z_1 \tp{\,W_1} )^{-1} = 1_n - \frac{1}{1+\<W_1,Z_1\>}\, Z_1 \tp{\,W_1}
\end{equation}
if $1+\< W_1, Z_1 \> \ne 0$.
\end{lemma}
\begin{proof}
The proof is straightforward, and is omitted.
\end{proof}


\bibliographystyle{amsplain}


\begin{thebibliography}{10}

\bibitem{BiquardGauduchon96}
O. Biquard and P. Gauduchon, \emph{La m\'etrique hyperk\"ahl\'erienne
  des orbites coadjointes de type sym\'etrique d'un groupe de {L}ie complexe
  semi-simple}, C. R. Acad. Sci. Paris S\'er. I Math. \textbf{323} (1996),
  no.~12, 1259--1264. \MR{1428547}

\bibitem{Calabi79}
E. Calabi, \emph{M\'etriques k\"ahl\'eriennes et fibr\'es holomorphes}, Ann.
  Sci. \'Ecole Norm. Sup. (4) \textbf{12} (1979), no.~2, 269--294. \MR{543218}

\bibitem{Donaldson02_holomorphic_discs}
S.~K. Donaldson, \emph{Holomorphic discs and the complex {M}onge-{A}mp\`ere
  equation}, J. Symplectic Geom. \textbf{1} (2002), no.~2, 171--196.
  \MR{1959581}

\bibitem{EguchiHanson79}
T. Eguchi and A.~J. Hanson, \emph{Self-dual solutions to {E}uclidean
  gravity}, Ann. Physics \textbf{120} (1979), no.~1, 82--106. \MR{540896}

\bibitem{equiv_sympl}
T. Hashimoto, \emph{A twisted moment map and its equivariance}, Tohoku
  Math. J. (2) \textbf{66} (2014), no.~4, 563--581. \MR{3868069}

\bibitem{Hitchin_et_al_87}
N.~J. Hitchin, A.~Karlhede, U.~Lindstr\"om, and M.~Ro\v{c}ek,
  \emph{Hyper-{K}\"ahler metrics and supersymmetry}, Comm. Math. Phys.
  \textbf{108} (1987), no.~4, 535--589. \MR{877637}

\bibitem{Hitchin87_monopole}
N. Hitchin, \emph{Monopoles, minimal surfaces and algebraic curves},
  S\'eminaire de Math\'ematiques Sup\'erieures [Seminar on Higher Mathematics],
  vol. 105, Presses de l'Universit\'e{} de Montr\'eal, Montreal, QC, 1987.
  \MR{935967}

\bibitem{Hitchin92_bourbaki}
\bysame, \emph{Hyper-{K}\"ahler manifolds}, no. 206, 1992, S\'eminaire
  Bourbaki, Vol. 1991/92, pp.~Exp. No. 748, 3, 137--166. \MR{1206066}

\bibitem{KapustinSaulina09}
A. Kapustin and N. Saulina, \emph{Chern-{S}imons-{R}ozansky-{W}itten
  topological field theory}, Nuclear Phys. B \textbf{823} (2009), no.~3,
  403--427. \MR{2571746}

\bibitem{Kronheimer90}
P.~B. Kronheimer, \emph{A hyper-{K}\"{a}hlerian structure on coadjoint orbits
  of a semisimple complex group}, J. London Math. Soc. (2) \textbf{42} (1990),
  no.~2, 193--208. \MR{1083440}

\end{thebibliography}
%

\providecommand{\bysame}{\leavevmode\hbox to3em{\hrulefill}\thinspace}
\providecommand{\MR}{\relax\ifhmode\unskip\space\fi MR }
\providecommand{\MRhref}[2]{%
  \href{http://www.ams.org/mathscinet-getitem?mr=#1}{#2}
}
\providecommand{\href}[2]{#2}

\end{document}